\documentclass[11pt]{article}
\usepackage{amsthm}
\usepackage{amssymb}
\usepackage{epsfig}
\usepackage{epstopdf}
\usepackage{amsthm}
\usepackage{srcltx}
\usepackage[dvipsnames,usenames]{color}
\usepackage{paralist}
\usepackage{amsmath}
\usepackage{amsfonts}
\usepackage{float}
\usepackage{xcolor}
\usepackage{soul}
\usepackage{lpic}

\usepackage{mathabx} 

\newcounter{contador}
\newcounter{teoA}

\newtheorem{teoa}[teoA]{Theorem}

\newtheorem{propo}[contador]{Proposition}
\newtheorem{teo}[contador]{Theorem}
\newtheorem{lem}[contador]{Lemma}

\newtheorem{corol}[contador]{Corollary}

\theoremstyle{remark}

\newcounter{ex}

\newcommand{\R}{{\mathbb R}}
\newcommand{\C}{{\mathbb C}}
\newcommand{\N}{{\mathbb N}}

\title{On some rational piecewise linear rotations}

\author{Anna Cima$^{(1)}$, Armengol Gasull$^{(1,2)}$, V\'{\i}ctor Ma\~{n}osa$^{(3)}$ and Francesc Ma\~{n}osas$^{(1)}$
    \\*[.1truecm]
    {\small \textsl{$^{(1)}$ Departament de Matem\`{a}tiques, Facultat
            de Ci\`{e}ncies,}}
    \\*[-.25truecm] {\small \textsl{Universitat Aut\`{o}noma de Barcelona,}}
    \\*[-.25truecm] {\small \textsl{08193 Bellaterra, Barcelona,
    Spain}}
        \\*[-.25truecm] {\small \textsl{anna.cima@uab.cat,
            armengol.gasull@uab.cat, manyosas@mat.uab.cat}}\\
 \\*[-.25truecm] {\small \textsl{$^{(2)}$ Centre de Recerca Matem\`{a}tica, Campus de
Bellaterra,}}
    \\*[-.25truecm] {\small \textsl{08193 Bellaterra, Barcelona, Spain}}\\
    \\*[-.25truecm] {\small \textsl{$^{(3)}$ Departament de Matem\`{a}tiques,}}
     \\*[-.25truecm] {\small \textsl{Institut de Matem\`{a}tiques de la UPC-BarcelonaTech (IMTech),}}           
    \\*[-.25truecm] {\small \textsl{Universitat Polit\`{e}cnica de Catalunya}}
    \\*[-.25truecm] {\small \textsl{Colom 11, 08222 Terrassa, Spain}}
    \\*[-.25truecm] {\small \textsl{victor.manosa@upc.edu}}}

\date{}

\begin{document}

\maketitle
\begin{abstract}   
We study the dynamics of the piecewise planar rotations $F_{\lambda}(z)=\lambda (z-H(z)),
$ with $z\in\C$, $H(z)=1$ if $\mathrm{Im}(z)\ge0,$ $H(z)=-1$ if $\mathrm{Im}(z)<0,$ and $\lambda=\mathrm{e}^{i \alpha} \in\C$, being $\alpha$ a rational multiple of $\pi$. Our main results establish the dynamics in the so called regular set, which is the complementary of the closure of the set formed by the preimages of the discontinuity line. We prove that any connected component  of this set is open, bounded and periodic under  the action of $F_\lambda$, with a period   $\ell,$ that depends on the connected component. Furthermore, $F_\lambda^\ell $ restricted to each component acts as a rotation with a period  which also depends on the connected component. As a consequence, any point in the regular set is periodic. Among other results, we also prove that for any connected component of the regular set, its boundary 
is a convex polygon with certain maximum number of sides.
\end{abstract}

%
%



\newpage

\section{Introduction and main results}

We consider the family of planar piecewise linear maps which, in complex notation, writes as:
$$
F_{\lambda}(z)=\lambda (z-H(z)),
$$ where $z\in\C$, $\alpha\in\R,$ $\lambda=\mathrm{e}^{i\alpha} \in\C$ (thus $|\lambda|=1$), and 
$$H(z)=\left\{
                          \begin{array}{ll}
                            1, & \hbox{if $ z\in \C^+_0$,} \\
                            -1, & \hbox{if $z\in \C^-$,}
                          \end{array}
                        \right.
$$
being $\C^+=\{z\in \C; \mathrm{Im}(z)> 0\},$ $\C^-=\{z\in \C; \mathrm{Im}(z)< 0\}$  and $\C^+_0=\{z\in \C; \mathrm{Im}(z)\ge 0\}.$  Observe that these maps are invertible. Indeed, some easy computations
show that $F_{\lambda}^{-1}(z)={z}/{\lambda}+H\left({z}/{\lambda}\right).$

These maps have been studied in \cite{BG,ChGQ,CGMM, GQ}. When $\alpha$ is a rational multiple of $\pi$, they are closely related to polygonal dual billiards maps on regular polygons, \cite{DT, GuSi, VS}. 
The special cases $\alpha\in\mathcal{R}:= \{{\pi}/{3},\pi/{2},2\pi/{3},4\pi/3,3\pi/2, 5\pi/3\}$ have been studied in \cite{ChaChe13,ChaChe14,ChaWanChe12,CGMM}. We found these examples to be especially  interesting because they are easy explicit examples of pointwise periodic bijective maps which are not globally periodic, hence their sets of periods are unbounded. Furthermore, for each of these cases, in \cite{CGMM}, we gave an explicit first integral  whose  energy levels are
discrete and  they are bounded sets 
whose interior is a necklace formed by a finite number of open tiles of a certain
regular tessellation. The boundary of each of these open regular tessellations is formed by the so-called \emph{critical set} $\mathcal{F}=\{z\in\mathbb{C}$ such that $\mathrm{Im}(F^i(z))=0$ for some $i\in\mathbb{N}\cup\{0\}\}$,
formed by all the preimages of the discontinuity    line  $\R.$  In general,  it is well known that this discontinuity line, which is also called critical line, together with its preimages play a crucial role to understand the dynamics of the corresponding map, see    \cite{BGM, MGBC96}.

The general properties of the maps $F$ with
$\alpha\in [0,2\pi)\setminus\mathcal{R}$, being a rational multiple of~$\pi$, are still not completely known.  For instance, in 
\cite{GQ}  it is proved that for each of such cases there exists a sequence of open invariant nested necklaces that tend to infinity,  whose beads are polygons,  and where the dynamics of $F$ is given by a product of two rotations, see also Section \ref{ss:boundedness}.  Remarkably, although the adherence of the union of all these invariant necklaces does not fill the full plane, it allows to prove that all orbits of $F$ are bounded. As in the regular cases, the boundary of these necklaces is given by some segments of the critical set $\mathcal{F}$.
Our simulations indicate that, in the non-regular cases, the critical set seems to fractalize.

We consider the \emph{regular set}  $\mathcal{U}=\mathbb{C}\setminus \overline{\mathcal{F}}$, where $\overline{\mathcal{F}}$ is the closure of the critical set. Among the results in this work,  we prove that when $\alpha$ is a rational multiple of $\pi,$ any connected component of $\mathcal{U}$ is open, bounded and periodic. Moreover, any element of $\mathcal{U}$ is periodic. 
We also prove that if $\overline{\mathcal{F}}\setminus\mathcal{F}\neq\emptyset$, then the elements of 
$\overline{\mathcal{F}}\setminus\mathcal{F}$ are aperiodic.

 In next two theorems the map $F_\lambda$ and the set $\mathcal U$ are defined as above. Our first result characterizes the dynamics on the regular set  when $\alpha$ is a rational multiple of $\pi$.

\begin{teoa}\label{t:teoa} Set $\alpha=2\pi{p}/{q}$  where $p,q\in\mathbb{N}$ with $(p,q)=1$, then any connected component  of $\mathcal{U}$ is open, bounded and periodic under the action of $F_\lambda$. Furthermore, $F_\lambda$ permutes $\ell$ connected components 
of $\mathcal{U}$, that are invariant by
$F_\lambda^\ell$, which is a rotation of
order~$k$ around the center of each connected component. As a
consequence, on each connected component there is an $\ell$-periodic point (the center)
and the rest of the points are $k\ell$-periodic.  Moreover, both values $k$ and $\ell$ depend on the connected component and $\ell$ is unbounded.
\end{teoa}

 Our second result  describes the geometry of the connected component of~$\mathcal{U},$   again when $\alpha$ is a rational multiple of $\pi$.

\begin{teoa}\label{t:teob} Set $\alpha=2\pi {p}/{q}$ where $p,q\in\mathbb{N}$ with $(p,q)=1,$ then:
\begin{enumerate}[(a)]
\item Let $V$ be a connected component of $\mathcal{U}.$
Then $\partial V$ is a \emph{convex polygon} with at most $q$ sides if $q$
is even and at most $2q$ sides when $q$ is odd. 
\item   If $\ell$ is the period of $V$ and $(\ell,q)=1$, then $\partial V$ has either $q$ sides and $\partial V$ is a \emph{regular polygon} or $q$ is odd and
$\partial V$ has $2q$ sides.
\item For some values of $\alpha$ there are connected components that are not regular polygons.
\end{enumerate}
\end{teoa}

Although in this work, we only deal with the case that $\alpha=2\pi \theta$ with $\theta$ rational, we remark that when  $\theta$ is irrational similar arguments that the ones used in the proof of Theorem \ref{t:teob} can be applied to prove than then $\partial V$ is a circle. We do not include this study in this work.

 Theorem \ref{t:teob} characterizes the connected components in the regular set $\mathcal{U}$ as open convex polygons. For this reason sometimes we will refer them as \emph{tiles} of the tessellation in $\mathcal{U}$ defined by $\overline{\mathcal{F}}$.

Theorem \ref{t:teoa} is proved in Section \ref{ss:proofofTA}
and Theorem \ref{t:teob} is proved in Section~\ref{ss:proofofTB}. In Section~\ref{s:critical} it is also proved that any point in $\mathcal{F}\setminus \mathcal{\overline{F}}$ is not periodic, see Proposition \ref{p:coros}.

In Section \ref{s:evidences}
we also present some evidences of the fractalization of $\mathcal{F}$, and the  unboundedness of periods in compact sets both in $\mathcal{U}$ and $\mathcal{F}$, as well as for the existence of non-periodic points in $\mathcal{F}$.

\section{Dynamics on the regular set}\label{s:proofteob}

The object of this section is to prove Theorem \ref{t:teoa}, see Section \ref{ss:proofofTA}. The essential ingredients of the proof are the facts that the connected components 
of the regular set $\mathcal{U}=\C\setminus \mathcal{\overline{F}}$ are the sets of points sharing the same itinerary (see the definition below) which are convex and bounded as well as the union of all its iterates. 

\subsection{Boundedness of orbits and connected subsets of the regular set}\label{ss:boundedness}

We are interested in the dynamical study of $F_{\lambda}$ in the
case when $\lambda=\mathrm{e}^{2\pi \frac{p}{q} i}$ with $p,q\in \N$ and
$(p,q)=1.$ That is, when the maps $F_{\lambda}^+:=F_{\lambda}\big\vert_{\C^+_0}$ and $F_{\lambda}^-:F_{\lambda}\big\vert_{\C^-}$
are (rational) periodic rotations. This situation was studied in
\cite{GQ}, where the authors proved  several facts which are essential for our
purposes. We list three consequences of their results. The first one, is that for any bounded subset $A$ of $\C$
the  (forward or backward) orbit of $A$ by $F_{\lambda}$ is
bounded. The second one, is that any connected unbounded set must
intersect $\mathcal{F}$ and the third one is that there is a sequence of invariant necklaces formed by periodic points and these necklaces  have increasing  periods when they tend towards infinity. We state these results for future references in the next result, which is a consequence of Theorem 1 in \cite{GQ}:

\begin{teo}[Goetz \& Quas, \cite{GQ}]\label{t:Goetz}  Set $\alpha=2\pi {p}/{q}$  where $p,q\in\mathbb{N}$ with $(p,q)=1$.
The following assertions hold:
\begin{enumerate}\item[(a)] For any bounded subset $A\subset \C$ the
orbit of $A$ is bounded. That is, the set
$\bigcup_{i=-\infty}^{\infty}F_{\lambda}^i(A)$ is bounded
\item[(b)] If $B$ is  a connected and unbounded subset of $\C$, then it must intersect
$\mathcal{F}.$

\item[(c)] There is a sequence of necklaces, like the ones in the right picture of Figure \ref{f:1}, which tend to infinity, filled with periodic point of $F_\lambda$ and with certain computable periods that tend to infinity.
\end{enumerate}\end{teo}

\begin{figure}[H]
\centerline{\includegraphics[scale=0.32]{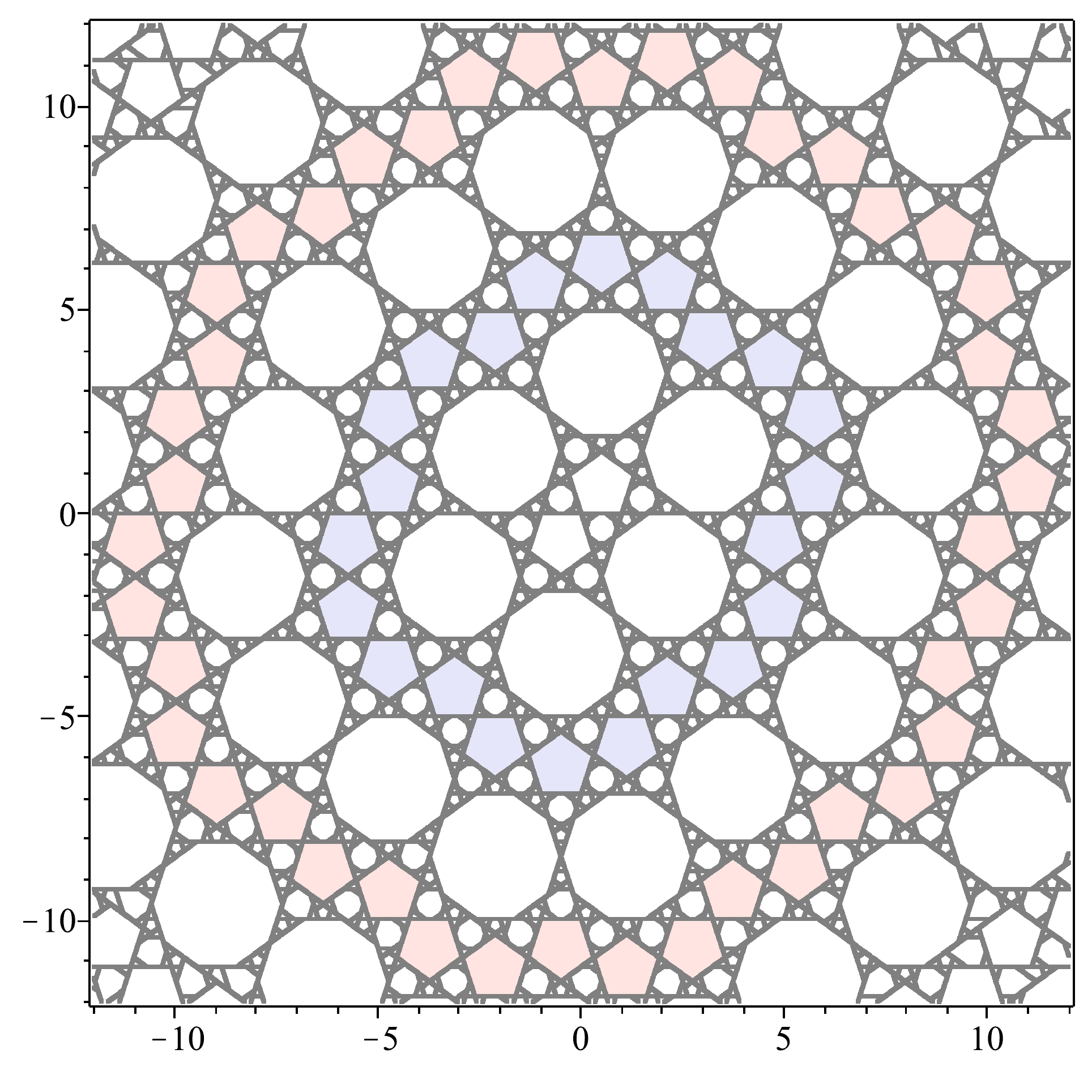}\qquad\includegraphics[scale=0.32]{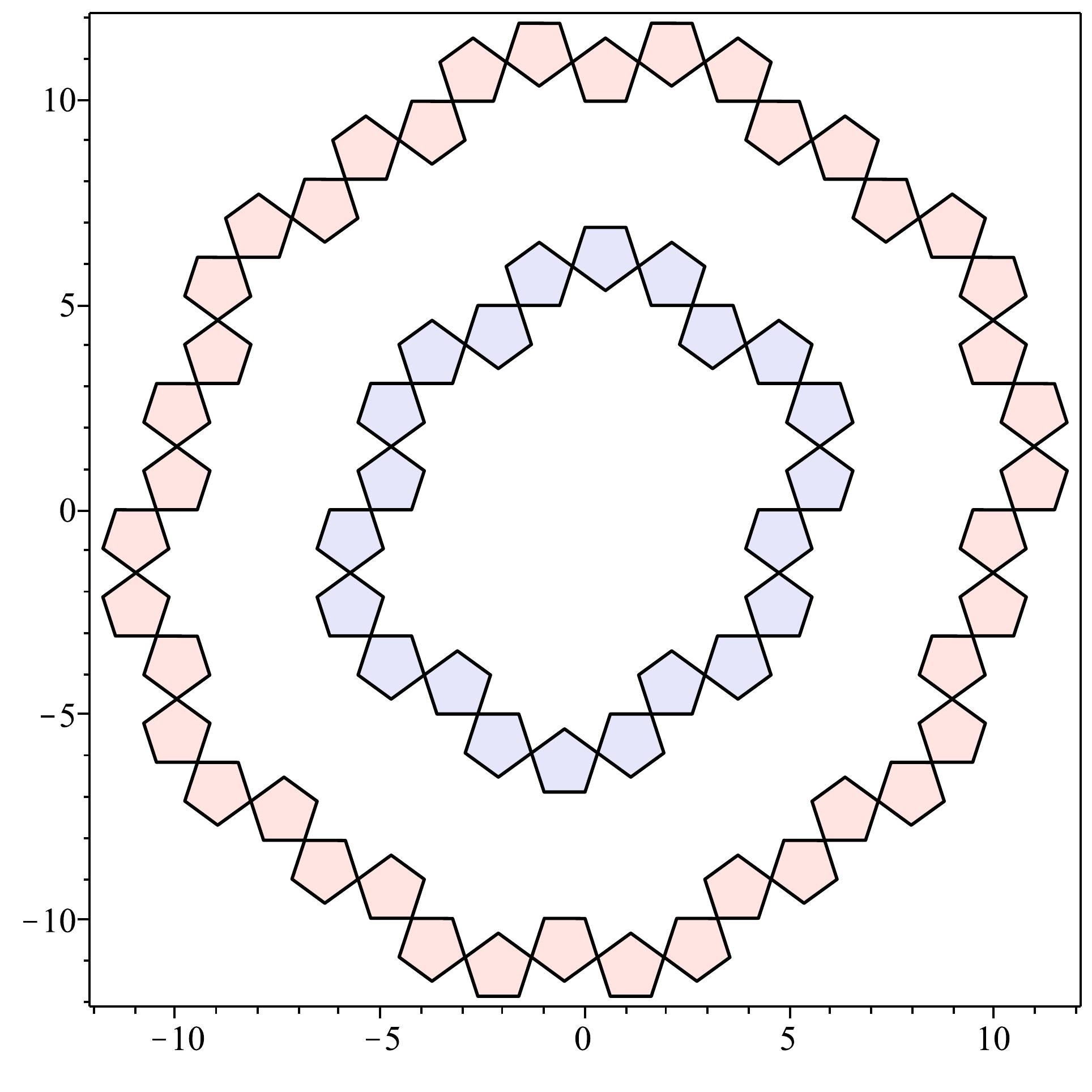}}
\caption{Necklaces for  $\alpha={8\pi}/{5}$. The critical set, formed by the union of a numerable sets of lines, in grey, in the left picture.  In the right picture two invariant necklaces.}\label{f:1}
\end{figure}

\subsection{Properties of the sets of points with the same itinerary}\label{ss:itineraries}

For $z\in \C$, we define the {\it address} of $z, $ denoted $A(z),$
by $$A(z)=\left\{
            \begin{array}{ll}
              +, & \hbox{if $z\in\C^+$,} \\
              -, & \hbox{if $z\in\C^-$.}
            \end{array}
          \right.$$

We also define the {\it itinerary of length $n$} as the finite
sequence of symbols $$\underline I_n(z)=A(z)A(F(z))\ldots
A(F^{n-1}(z)),$$ and the {\it itinerary} as the infinite sequence of
symbols $$\underline I(z)=A(z)A(F(z))\ldots A(F^{i}(z))\ldots$$

Notice that if $\underline I_n(z)=s_1\ldots s_n$ then
$F_{\lambda}^n(z)=F_{\lambda}^{s_n}\circ\ldots\circ
F_{\lambda}^{s_1}(z)$.

Let ${\cal J}$ be the set of infinite sequences of two symbols $+$
and $-,$ and let $S$ be the shift operator defined in $\cal J.$ That is
$S(s_1s_2\ldots s_i\ldots)=s_2s_3\ldots s_i\ldots.$ An element
$\underline I\in \cal J$ is called $n$-periodic if $S^n(\underline
I)=\underline I$ and $n$ is the smallest natural with this property.
In this case we will write $\underline I=(s_1s_2\ldots
s_{n})^{\infty}.$ Since $\underline
I(F_{\lambda}^n(z))=S^n(\underline I(z))$ if $z$ is $n$-periodic for
$F_{\lambda}$ then $\underline I(z)$ is $k$-periodic with $k$
divisor of $n.$

Let ${\cal J}_n$ be the set of finite sequences of two symbols $+$
and $-,$ with length $n.$ For $\underline J\in {\cal J}$ and
$\underline J_n\in {\cal J}_n$ we define the corresponding subsets
of $\C$ of points in the regular set with the same itinerary, given by
$$B(\underline J)=\{z\in \mathcal{U}\,:\, \underline I(z)=\underline J\} \mbox{ and } B(\underline J_n)=\{z\in \mathcal{U}\,:\, \underline I_n(z)=\underline
J_n\}.$$

The next result can be found in \cite[Lemma 3]{CGMM}, we include its proof for completeness.
\begin{lem}\label{l:convex} (i) Let $\underline J_n\in {\cal J}_n.$ Then $B(\underline
J_n)$ is either empty or convex. Moreover
$F_{\lambda}^n\big\vert_{B(\underline J_n)}$ is an affine map.

(ii) Let $\underline J\in {\cal J}.$ Then $B(\underline
J)$ is either empty or convex and bounded.
\end{lem}

\begin{proof} The proof of the convexity follows easily by induction. If $n=1,$
$B(\underline J_n)$ is either $\mathbb{C}^+$ or $\mathbb{C}^-$ both convex sets.
Assume that the result holds for sequences of length $n-1$ and set
$\underline J_{n-1}=(s_1,\ldots,s_{n-1}).$ Therefore we have
$$B(\underline J_n)=\{(x,y)\in B(\underline J_{n-1})\,:\, F_\lambda^{n-1}(x,y)\in \mathbb{C}^{s_{n}}\}.$$ Moreover,
 $F_\lambda^{n-1}$ restricted to $B(\underline J_{n-1})$ is the affine map
$G=F_{s_{n-1}}\circ\ldots\circ F_{s_1}.$ So we have
$$B(\underline J_n)=B(\underline J_{n-1})\cap G^{-1}(\mathbb{C}^{s_{n}}).$$ This fact
 proves that $B(\underline J_n)$ is convex because it is
 the intersection of two convex sets. This ends the inductive proof of convexity. Furthermore,
 $F_\lambda^n(x,y)=F_\lambda^{s{_n}}\circ F_\lambda^{s_{n-1}}\circ\cdots\circ
F_\lambda^{s_1}(x,y),$ for all $(x,y)\in B(\underline J_n),$ showing that
$F_\lambda^n$ restricted to $B(\underline J_n)$ is an affine map.

\smallskip

(ii) Write $\underline J=s_1s_2\ldots s_i\ldots$ Then
$B(\underline J)=\bigcap_{n=1}^{\infty}B((s_1\ldots s_n))$ and the lemma
follows because the infinite intersection of convex sets is either
empty or convex. Since it is convex it is also connected. Then,  it
is bounded  from Theorem \ref{t:Goetz} (b) (recall that, by definition, $B(\underline J)\subset\mathcal{U}$).
\end{proof}

The next result is one of the key steps in the proof of Theorem \ref{t:teoa}. It establishes that for $\alpha$ being a rational multiple of $\pi$, the points with periodic itinerary are periodic:

\begin{propo}\label{p:itiper} Assume that $\lambda=e^{\frac{p}{q}2\pi i}$ with $p,q\in\mathbb{N}$ with $(p,q)=1$.  Let $z\in\mathcal{U},$ then $\underline I(z)$ is periodic if and only if $z$ is periodic for
$F_{\lambda}.$ Moreover $B(\underline I(z))$ is bounded and convex.\end{propo}

\begin{proof} Trivially, if $z$ is a periodic point then $\underline I(z)$ is periodic.  Assume now, that  $z\in\mathcal{U},$ is a point such that 
$\underline I(z)=(s_1\ldots s_\ell)^{\infty}.$ In this case
$B(\underline I(z))$ is non empty and positively invariant by
$F_{\lambda}^\ell.$ By Lemma \ref{l:convex} (i), the map
$F_{\lambda}^\ell\big\vert_{B(\underline I(z))}$ is affine. Direct
computations show that $F_{\lambda}^\ell\big\vert_{B(\underline
I(z))}(z)=\lambda^\ell z+b$ for some $b\in \C.$ Thus when $\ell$ is not a
multiple of $q$ we have that $F_{\lambda}^\ell\big\vert_{B(\underline
I(z))}$ is a rotation of angle $\ell 2\pi p/q\ne 0$ centered at some
point of the plane. This shows that all points belonging to
$B(\underline I(z))$ are periodic. Note that, in fact, if $w\in
B(\underline I(z))$, then the orbit by the rotation of angle $\ell
2\pi p/q$  is also contained in $B(\underline I(z)).$ Since, by Lemma~\ref{l:convex}~(ii), $B(\underline I(z))$ is convex it follows that the
center of the rotation is contained in $B(\underline I(z)).$ On the
other hand  if $B(\underline I(z))$ is unbounded, from Lemma~\ref{t:Goetz} it must intersects $\mathcal{F}\,:\,$ a contradiction.

When $\ell$ is a multiple of $q$ we get
$F_{\lambda}^\ell\big\vert_{B(\underline I(z))}(z)=z+b.$ If $b\ne 0$ it
follows that $B(\underline I(z))$ must be unbounded\,:\, in
contradiction with Lemma \ref{l:convex} (ii). So $b=0$ and $F_{\lambda}^\ell\big\vert_{B(\underline
I(z))}= \mathrm{Id}.$ 
\end{proof}

\subsection{Proof of Theorem \ref{t:teoa}}\label{ss:proofofTA}

To end the proof of Theorem \ref{t:teoa} we need the following last technical result.

\begin{lem}\label{l:con} If $V\subset\mathbb{C}$ is connected and $V\cap\mathcal{F}=\emptyset,$ then all points in $V$ have the same itinerary.\end{lem}

\begin{proof} Since  $V\cap\mathcal{F}=\emptyset,$ then $F_{\lambda}^i(V)\cap\R=\emptyset$ for all $i\in\N.$
Therefore $F_{\lambda}^i\big\vert_V$ is continuous and since $V$ is
connected  $F_{\lambda}^i(V)$ is contained in one of the two
connected components of $\C\setminus\R.$ This ends the proof of the
lemma.\end{proof}

\begin{proof}[Proof of Theorem \ref{t:teoa}] The fact that any component component of $\mathcal{U}$ is open follows from the fact that $\mathcal{U}$ is open. Moreover, since any connected
component does not intersects $\mathcal{F}$ it follows, by Theorem \ref
{t:Goetz} (b), that they are bounded. Also  since any connected
component $V$ does not intersect $\mathcal{F}$, it follows that
$F^i_{\lambda}\big\vert_V$ is continuous and preserves area for  any $i\in\N$. So we obtain that for all
$i\in\N$ the sets $F_{\lambda}^i(V)$ are connected and with the same
area that~$V$. Since, by Theorem \ref{t:Goetz} (a),
$\bigcup_{i=0}^{\infty}(V)$ is bounded, then there must be an overlapping of the images of $V$, that is, there exists
$n,\ell\in \N$ such that 
\begin{equation}\label{e:overlapping}
F_{\lambda}^n(V)\cap
F_{\lambda}^{n+\ell}(V)\ne\emptyset.
\end{equation} 
By Lemma \ref{l:con}, the points in
$F_{\lambda}^n(V)$ have all the same itinerary, namely $\underline
J.$ As a consequence of~\eqref{e:overlapping},
necessarily $S^\ell(\underline J)=\underline J$, that is, the itinerary  
$\underline J$ is $\ell$-periodic.
  Therefore, from Proposition
\ref{p:itiper}, the points in $F_{\lambda}^n(V)$ are
periodic and, since $F_{\lambda}$ is invertible, the same holds for the
points in~$V.$ 

 Notice that $F_\lambda$  permutes the $\ell$ tiles  $F^i_\lambda(V)$ with $i=0,\ldots,\ell-1$. 
Moreover, each of these sets $F^i_\lambda(V)$ is invariant by
$F_\lambda^\ell(z)=\lambda^\ell z+b$, with   $b\in\mathbb{C}$, which is a rotation of
order~$k$ around a center point, which is contained in  this set  because of the convexity  of $F^i_\lambda(V)$; see Lemma~\ref{l:convex}~(ii) and the proof of Proposition \ref{p:itiper}. As a
consequence, on each tile there is a $\ell$-periodic point (the center)
and the rest of the points are $k\ell$-periodic. Of course $F_\lambda^{k\ell}=\mathrm{Id}$.

 Finally, the periods on $\mathcal{U}$ are unbounded as a direct consequence of the results of~\cite{GQ}, see item (c) in Theorem \ref{t:Goetz}.
\end{proof}

\section{Geometrical aspects.}\label{s:critical}

The objective of this section is to prove Theorem \ref{t:teob}, which describes the geometry of the boundary of the connected components of $\mathcal{U}$. This is done in Section \ref{ss:proofofTB}. To do this, several results characterizing the closure of the critical set must be established.

\subsection{Geometry of the critical set and regular set.}\label{ss:critical}
For any $z\in \C$, we denote the distance of $z$ to the  abscissa axis,  which is the critical line $\R,$  by $d(z,\R).$

\begin{lem}\label{l:equi} Let $\overline{\mathcal{F}}$ be the closure of the critical set, then $$\overline{\mathcal{F}}=\left\{z\in\C\,:\, \inf_{n\in\N
\cup\{0\}} d(F_{\lambda}^n(z),\R)=0\right\}.$$  \end{lem}

\begin{proof} Let us denote by $\mathcal{S}=\{z\in\C\,:\, \inf_{n\in\N \cup\{0\}}
d(F_{\lambda}^n(z),\R)=0\}$. First, we will show that
$\overline{\mathcal{F}}\subset \mathcal{S}.$ To do this we take $z\notin \mathcal{S}$ and
we will show that $z\notin \overline{\mathcal{F}}.$ If $z\notin \mathcal{S}$ it
follows that $\inf_{n\in\N \cup\{0\}} d(F_{\lambda}^n(z),\R)=a>0.$
Now consider $D$ the open disc centered at $z$ with radius $b<a.$ We
will show that $D\cap \mathcal{F}=\emptyset$ which implies that $z\notin
\overline{\mathcal{F}}.$ To do this we will prove inductively that $D\cap
F_{\lambda}^{-n}(\R)=\emptyset$ for all $n\in \N\cup \{0\}.$ This is
clear for $n=0$ because $d(z,\R)\geq a>b.$ Now assume that $D\cap
F_{\lambda}^{-i}(\R)=\emptyset$ for all $i\in\{0,\ldots,n\}.$ Then
it follows that $F_{\lambda}^{i}(D)\cap\R=\emptyset$ for all
$i\in\{0,\ldots,n\}.$ Observe that $F\vert_{F^n(D)}=F^{n+1}\vert_D$ is a rotation that
in particular preserves the distance, so if $F^{n+1}(D)\cap \R\ne
\emptyset$ then $d(F_{\lambda}^{n+1}(z),\R)<b<a$, which is a contradiction.
Then $F^{n+1}(D)\cap \R= \emptyset,$ and this ends the inductive
proof.

Now  we show that $\mathcal{S}\subset\overline{\mathcal{F}} .$ Assume that $z\notin
\overline{\mathcal{F}}$ and $z\in \mathcal{S}$. Then from Theorem \ref{t:teoa}, $z$ is
periodic. Therefore its orbit is finite and $\inf_{n\in\N \cup\{0\}}
d(F_{\lambda}^n(z),\R)=\min_{n\in\N \cup\{0\}}
d(F_{\lambda}^n(z),\R)=a.$ If $a=0$ then $z\in \mathcal{F}\,:\,$ a
contradiction. So $a>0$ and the disc $D$ centered at $z$ with radius
$b<a$ does not intersects $\mathcal{F}$, a contradiction, again.\end{proof}

We denote by $\mathcal{G}$  the critical set for $F_{\lambda}^{-1}.$
That is
$$\mathcal{G}=\bigcup_{i=1}^{\infty}F^{i}_{\lambda}(\mathbb{R}).$$

\begin{propo}\label{p:coros} The following assertions hold
\begin{enumerate}\item[(a)] If $\overline{\mathcal{F}}\setminus\mathcal{F}\ne \emptyset$ then the elements of
$\overline{\mathcal{F}}\setminus\mathcal{F}$ are aperiodic.
\item[(b)] $\mathcal{G}\subset \overline{\mathcal{F}}.$
\end{enumerate}\end{propo}
\begin{proof} (a) Assume, to get a contradiction, that $z\in
\overline{\mathcal{F}}\setminus\mathcal{F}$ is periodic. Therefore its orbit is
finite and $$\inf_{n\in\N \cup\{0\}}
d(F_{\lambda}^n(z),\R)=\min_{n\in\N \cup\{0\}}
d(F_{\lambda}^n(z),\R)=a.$$ From Lemma \ref{l:equi}, $a=0.$ But this
implies that $z\in\mathcal{F}$, a contradiction.

(b) Assume, to get a contradiction, that there exists $z\in
\mathcal{G}\cap \mathcal{U}$. By Theorem \ref{t:teoa},  $z$ is periodic, but due to the bijectivity of $F$, this also implies that $z\in\mathcal{F}$, a contradiction.
\end{proof}

\begin{corol}\label{c:inv} The regular set $\mathcal{U}$ and the closure of the critical set $\overline{\mathcal{F}}$ are invariant sets
(both positively and negatively) for $F_{\lambda}.$ Moreover
$F_{\lambda}\big\vert_\mathcal{U}$ is a pointwise periodic homeomorphism, which 
permutes the connected components of
$U.$
\end{corol}

\begin{proof} By definition $F_{\lambda}(\mathcal{U})\cap \mathcal{F}=\emptyset.$ On the other hand, as a consequence of Theorem \ref{t:teoa} and Proposition \ref{p:coros} (a), 
$F_{\lambda}(\mathcal{U})\cap
\left(\overline{\mathcal{F}}\setminus\mathcal{F}\right)=\emptyset.$ Then
$F_{\lambda}(\mathcal{U}) \subset \mathcal{U}.$  The same considerations holds for
$F_{\lambda}^{-1}$ and therefore $\mathcal{U}$ is positively and negatively
invariant. Since $\overline{\mathcal{F}}$  is the complement of $\mathcal{U},$ the same
holds for it. Lastly, since $\mathcal{U}$ is disjoint from the
critical set of $F_{\lambda}$ and $F_{\lambda}^{-1}$ it follows that
$F_{\lambda}\big\vert_\mathcal{U}$ and $F_{\lambda}^{-1}\big\vert_\mathcal{U}$ are both continuous maps.\end{proof}

The next result, characterizes each connected component of the regular set by its (periodic) itinerary:

\begin{corol}\label{c:dif} Let $V$ be a connected component of $\mathcal{U}.$ Let
$I$ be a finite-length itinerary such that for all $z\in V$  we get $\underline I(z)=I^{\infty}.$
Then $V=B(I^{\infty}).$ In particular the connected components of
$\mathcal{U}$ are convex.\end{corol}

\begin{proof} Clearly $V\subset B(I^{\infty}).$ On the other hand
$B(I^{\infty})\cap\overline{\mathcal{F}}=\emptyset,$ because by definition it does
not contains points in $\mathcal{F}$, and the points in
$\overline{\mathcal{F}}\setminus\mathcal{F}$ are not periodic (Proposition \ref{p:coros}). Therefore
$B(I^{\infty})\subset \mathcal{U}.$ By Lemma \ref{l:convex} (ii), the set $B(I^{\infty})$ is convex, so it is also connected and, in consequence, it is contained only in one connected component of $\mathcal{U}.$ Then
$B(I^{\infty})\subset V.$\end{proof}

The following Proposition summarizes some known results about convex
planar sets that we  are going to use, see for instance \cite[Part I, Chapter 1]{S}.

\begin{propo}\label{p:convex} Let $A$ be a convex open planar set and $\partial A$ its boundary. Then the
following assertions hold:
\begin{enumerate}
\item[(a)] For any point $z\in
A$ and any point $w\in \partial A$, the segment joining both points
is entirely contained in $A$, except for the point $w.$

\item[(b)]
$\partial A$ is the union of a countable set of ${\cal C}^1$-open
arcs and their endpoints.
\item[(c)] If $z$ belongs to one of the arcs in statement (b), then the tangent line to $\partial A$ at $z$ does not cut $A.$
\item[(d)] If $\partial A$ is a polygon, then it contains at most two
segments with the same slope.
\end{enumerate}
\end{propo}

\begin{lem}\label{l:front} The boundary of a connected component of $\mathcal{U}$ is contained in
$\mathcal{F}.$\end{lem}

\begin{proof} Let $V$ be a connected component of $\mathcal{U}$ and denote its
boundary by $\partial V.$ Since $V$ is open $\partial V\subset
\overline{\mathcal{F}}.$ Let $z\in \partial V$ and assume, to obtain a
contradiction, that   $z\in \overline{\mathcal{F}}\setminus\mathcal{F}.$ Pick a point $z'\in V$ and
consider $L$ be the segment that joints $z'$ and $z.$ By the
convexity of $V$ it follows that all points in $L$, except $z$,
belong to $V$ (Proposition \ref{p:convex} (a)). Since $z\notin \mathcal{F}$, by Lemma \ref{l:con} all points in $L$, including $z$, have the same itinerary. From
Theorem \ref{t:teoa}, it follows that $z'$ is periodic
and hence its itinerary is also periodic. So 
the itinerary of $z$ is periodic as well, and as a consequence, 
$z$ is periodic. This fact contradicts Proposition~\ref{p:coros} (a).~\end{proof}

\subsection{Proof of Theorem \ref{t:teob}}\label{ss:proofofTB}

\begin{proof}[Proof of Theorem \ref{t:teob}] (a) Let $z\in \partial V.$ From Lemma \ref{l:front}, $z$
belongs to $\mathcal{F}.$ Therefore there exists a first iterate $i$ such that
$z$ belongs to $F_{\lambda}^{-i}(\R).$ At this point note that
$F_{\lambda}^{-i}(\R)$ is a finite union of segments
all of them with the same slope. The argument of this slope belongs to the set
$\Theta=\left\{0,\frac{2\pi}{q},\ldots,\frac{(q-1)2\pi}{q}\right\}.$ So there
are $q$ possible slopes in the case $q$ odd and only $\frac {q}{2}$
in the even case. Note also that for each iterate $i$ the number of
endpoints of this segments is finite and then the set of points in
$\mathcal{F}$ that are not contained in the interior of one segment in
$\mathcal{F}$ is countable.

From Proposition \ref{p:convex} (b), $\partial V$ is a
countable union of ${\cal C}^1$-open arcs and their endpoints. We will show that each of
these arcs is a segment with slope in $\Theta.$ Pick a point $z$ in
the interior of one of these arcs not belonging to the countable set
of extremal points in $\mathcal{F}.$ In this point, the tangent line to $\partial V$ is well defined 
and it (locally) belongs to the interior of a
segment in $\mathcal{F}.$ It is clear that this segment must coincide
with the tangent. If not, this segment crosses $\partial V\,:\,$ a
contradiction. So almost all points (all except a countable set) have
a tangent with slope in $\Theta.$ Since the slope varies
continuously in each arc, it follows that it is constant and each
arc is a segment with slope in $\Theta.$ So $\partial V$ is a
polygon. Now the result about the number of sides follows from
Proposition \ref{p:convex} (d).

(b) If $(\ell,q)=1$ we get that $F^\ell_{\lambda}$ is a rotation with
order $q$ that leaves $V$ invariant. So the only possibility is that either $q$ is even and
$\partial V$ has $q$ sides, or $q$ is odd and $\partial V$ has $2q$
sides. In the first case it follows that each side has the same
length and also each angle of the polygon must be equal.

(c) In order to prove the existence of connected components of $\mathcal{U}$ which are non-regular polygons, we consider $\alpha=11\pi/6$ (that is, $\alpha=2\pi p/q$ with $p=11$ and $q=12$), and we consider the non-regular hexagon 
\begin{align*}
H:=& \Bigg \langle
\left(2,0\right),\, \left(\frac{\sqrt{3}+3}{2},0\right), \left(\frac{\sqrt{3}+3}{2},\frac{\sqrt{3}-1}{2}\right), \left(\frac{\sqrt{3}+7}{4},\frac{\sqrt{3}+1}{4}\right), \left(\frac{\sqrt{3}+2}{2},\frac{1}{2}\right), \\
&\left(\frac{\sqrt{3}+5}{4},\frac{\sqrt{3}-1}{4}\right)\bigg \rangle,
\end{align*}
 where $\langle\,\rangle$ stands for the convex hull.  The center of $H$ is the point $C=\left(\frac{\sqrt{3}}{3}+\frac{3}{2},\frac{\sqrt{3}}{6}\right)\in\mathcal{U}$, which is a $20$-periodic point.


Carefully keeping track the iterates of the segment joining the points $\left(2,0\right)$ and $\left(\frac{\sqrt{3}+3}{2},0\right)$ which belong to $\mathcal{F}$, we obtain that all the segments of the boundary of $H$, belong to $\mathcal{F}$.

By inspection, we have that the interior of this hexagon, $H^{\mathrm{o}}$, does not cuts the discontinuity line. The maps $F^+$ and $F^-$,  respectively, are rotations of angle $\alpha=11\pi/6$ around the centers of the upper and lower big dodecagons depicted in Figure \ref{f:2}, respectively. Hence, the images by $F$ of the interior of the hexagon, $H_n^{\mathrm{o}}=F^n(H^{\mathrm{o}})$, evolve describing a rotation around the dodecagon which corresponds with the address of the center of $H_n^{\mathrm{o}}$. As a consequence, one can check that $H_{20}^{\mathrm{o}}=H^\mathrm{o}$  and that for $n=0,\ldots,19$, $H_{n}^{\mathrm{o}}\cap \mathbb{R}=\emptyset$. As a consequence, these $20$ open irregular hexagons are an invariant subset of $\mathcal{U}$ under the action of $F$.\end{proof}

\begin{figure}[H]
	\centerline{\includegraphics[scale=0.34]{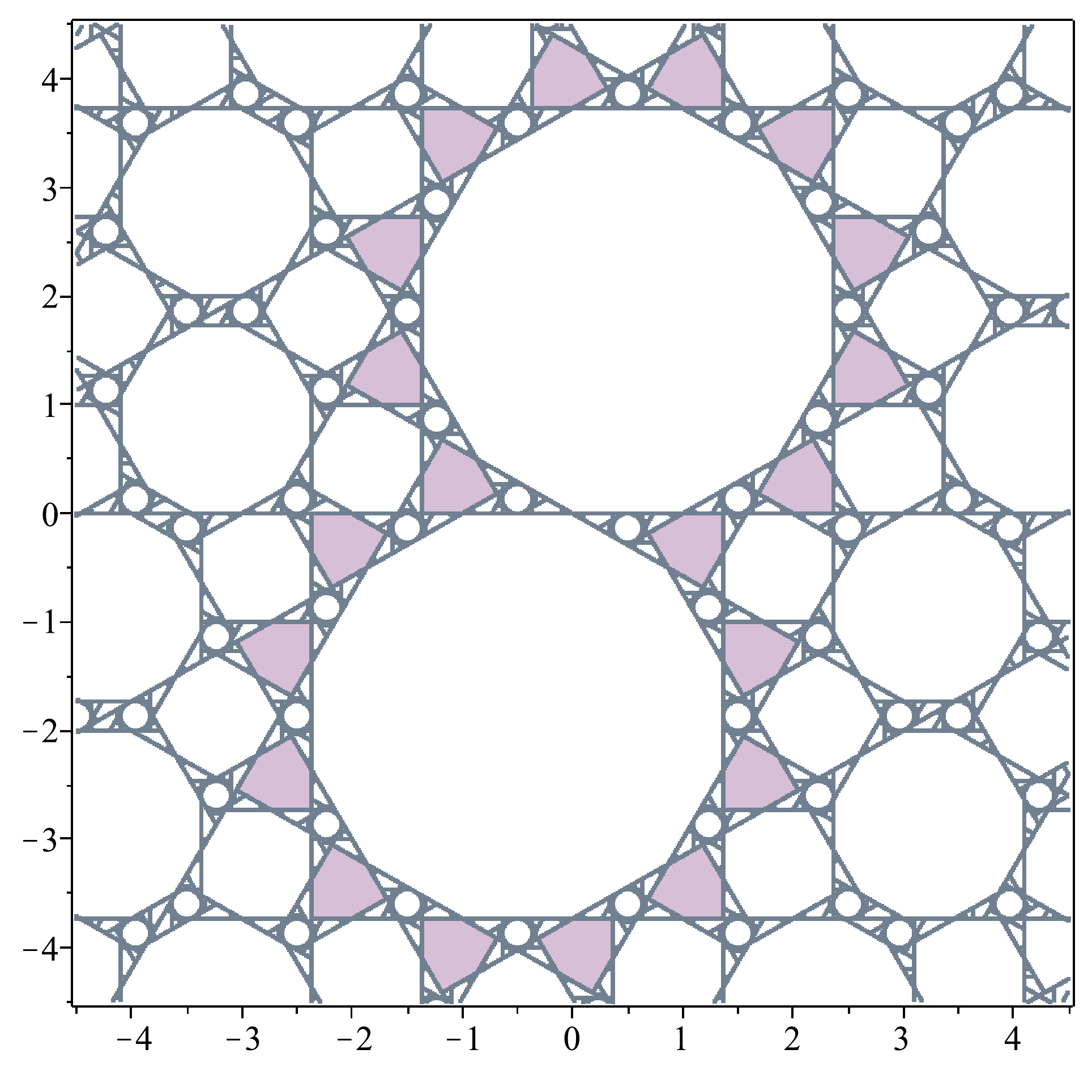}}
	\caption{The 20-periodic irregular hexagons associate to the hexagon $H$, for $\alpha=11\pi/6$.}\label{f:2}
\end{figure}

\section{Some further evidences.}\label{s:evidences}

One of the main differences between the special cases studied in \cite{BG,ChGQ,CGMM, GQ}, in which $\alpha\in\mathcal{R}= \{\pi/{3},\pi/{2},2\pi/{3},4\pi/3,3\pi/2, 5\pi/3\}$, and the cases
$\alpha=2\pi\frac{p}{q}\notin\mathcal{R}$, is the apparent fractalization of both $\mathcal{F}$ and $\mathcal{U}$ (see Figure \ref{f:3}, for instance), and the unboundedness of the periods of the periodic orbits that can be found in compact sets. In this final section, we present some particular evidences in this direction.

\begin{figure}[H]
\centerline{\includegraphics[scale=0.36]{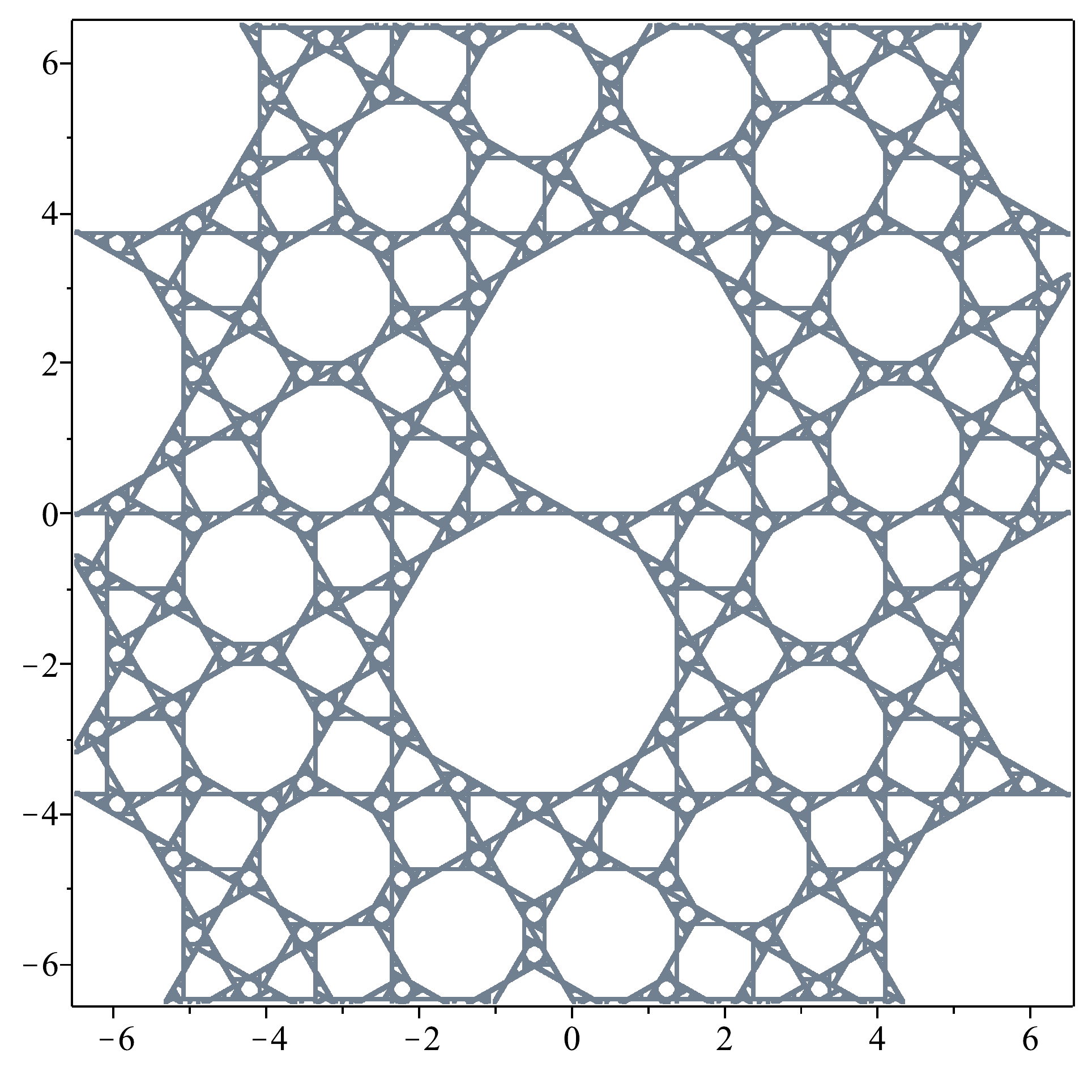}\quad \includegraphics[scale=0.36]{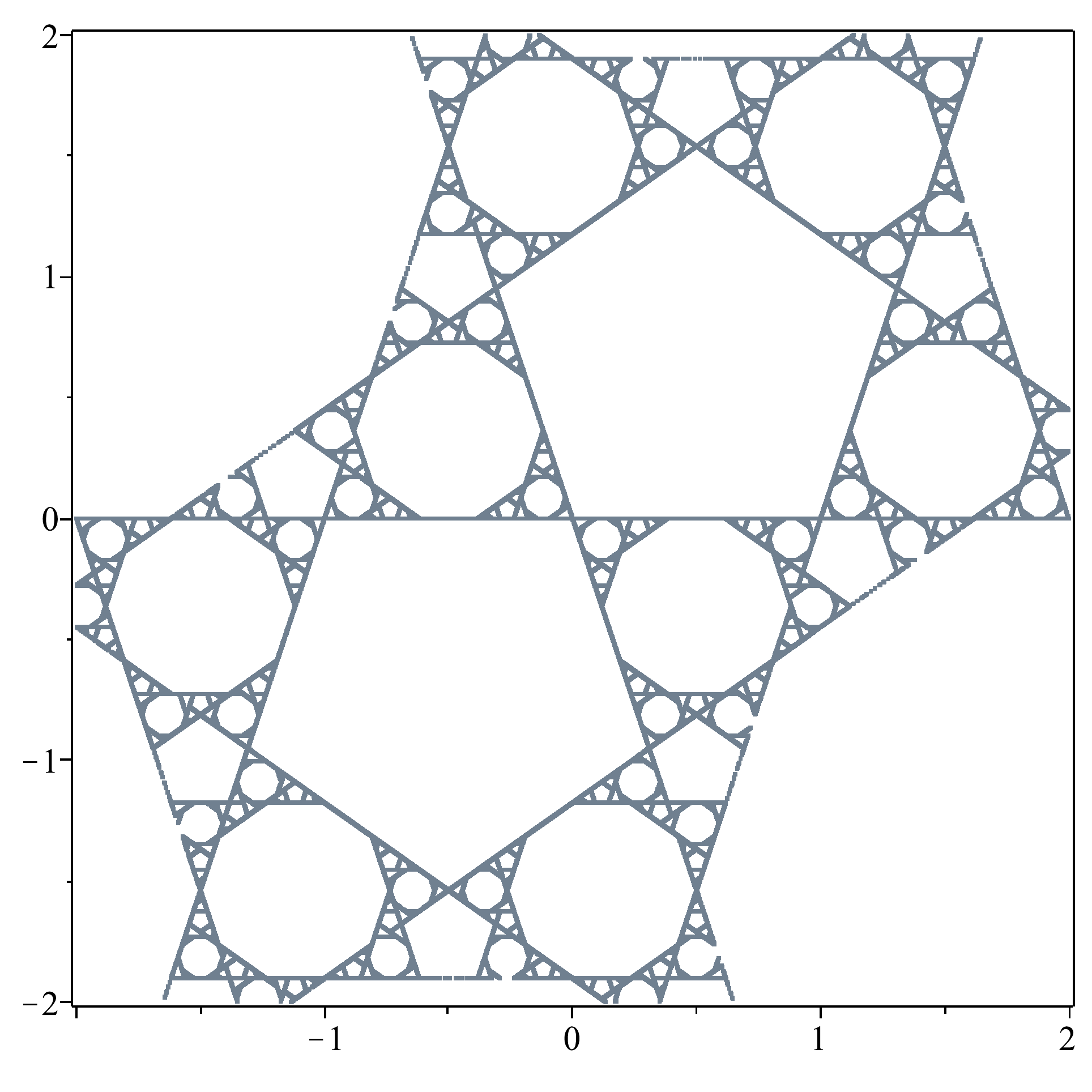}}
\caption{Apparent fractalization of $\mathcal{F}$ and $\mathcal{U}$ when $\alpha=2\pi\frac{p}{q}\notin\mathcal{R}$, for the cases $\alpha=\frac{11\pi}{6}$ and $\alpha=\frac{8\pi}{5}$, respectively.}\label{f:3}
\end{figure}

As en example, for $\alpha=8\pi/5$, by using geometrical arguments we found a scale factor of $1/\varphi^3$, where $\varphi=(1+\sqrt{5})/2$ is the golden ratio, between the triangle contained in $\mathcal{F}$, defined by the points
$$Q=(-\varphi,0),\, R=\left(\frac{1}{2},\frac{(1+2\varphi)\varphi\sqrt{\varphi+2}}{2}\right) \mbox{ and }S=(1+\varphi,0)$$
and, seemingly, two infinite sequence of nested triangles in both left and right directions. More precisely, the left sequence of nested triangles are obtained using the rescaling $r(x,y)=\left((2\varphi-3)x+2-2\varphi,(2\varphi-3)y\right)$ (notice that $1/\varphi^3=2\varphi-3$). Starting by $\triangle QRS$  and obtaining the sequence $\triangle QR_iS_i$ where  $R_i=r^i(R)$ and $S_i=r^i(S)$. See Figure \ref{f:4}.

\begin{figure}[H]
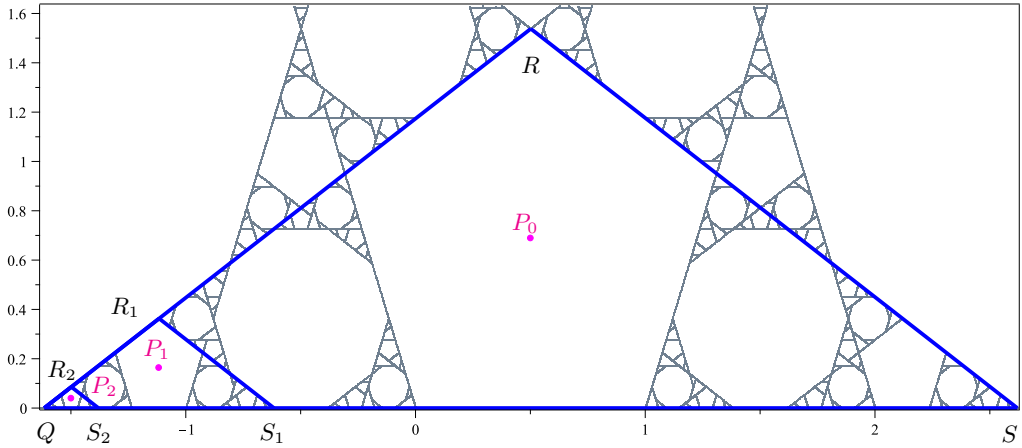

\footnotesize
  \centering
  \begin{lpic}[l(2mm),r(2mm),t(2mm),b(2mm)]{fractal1(0.7)} 
      \lbl[c]{100,45;{\color{magenta} $P_0$}}
     \lbl[c]{30,21;{\color{magenta} $P_1$}}
    \lbl[c]{20,14;{\color{magenta} $P_2$}}
    \lbl[c]{9,5; $Q$}
    \lbl[c]{101,75; $R$}
    \lbl[c]{24,29; $R_1$}
    \lbl[c]{12,17; $R_2$}
     \lbl[c]{192,5; $S$}
       \lbl[c]{52,5; $S_1$}
         \lbl[c]{19,5; $S_2$}
    \end{lpic}
  \caption{Sequence of nested triangles defined by the critical curves for $\alpha=\frac{8\pi}{5}$. In blue, $\triangle QRS$, $\triangle QR_1S_1$ and $\triangle QR_2S_2$. In magenta, the periodic points $P_0,P_1$ and $P_2$.}\label{f:4}
\end{figure}

The rescaling $r$ allows to obtain a seemingly infinite sequence of periodic points with unbounded periods in a compact set of $\mathcal{U}$.  Indeed, if we consider the fixed point $P_0=\left(\frac{1}{2},\frac{1}{10}\sqrt{\left(2+\varphi \right)^{3}}\right)$, which is the center of the pentagon in Figure \ref{f:4}, and we consider the recurrence
$P_{n+1}=r(P_{n})$, we obtain a sequence of centers of nested pentagons whose periods seems to monotonically increase, see also Figure \ref{f:4}.

\begin{table}[h]
	\centering
	\begin{tabular}{|c|c|c|c|c|c|c|c|c|c|c|}
		\hline
		$P_n=r^n(P_0)$ & $P_0 $ & $P_1$ & $P_2 $ & $P_3$ & $P_4 $ & $P_5$ & $P_6 $ & $P_7$ & $P_8 $ & $P_9$\\
		\hline
		Period & $1$ & $7$ & $38$& $232$  & $1338$ & $8332$ & $49988$ &$299932$ & $1799588$ & $10797532$\\
		\hline
	\end{tabular}
	\caption{First periods of the centers of the pentagons in Figure \ref{f:4}.}\label{t:1}
\end{table}

For instance, $P_1=r(P_0)$ is the center of a second pentagon and it is $7$-periodic, which induces a $7$-periodic inter-tile dynamics. Its itinerary map is a $5$-order rotation, hence there exists $7$ pentagons filled by $35$-periodic orbits. See Figure \ref{f:6}. The first periods we encounter are given in Table \ref{t:1}


\begin{figure}[H]
\centerline{\includegraphics[scale=0.35]{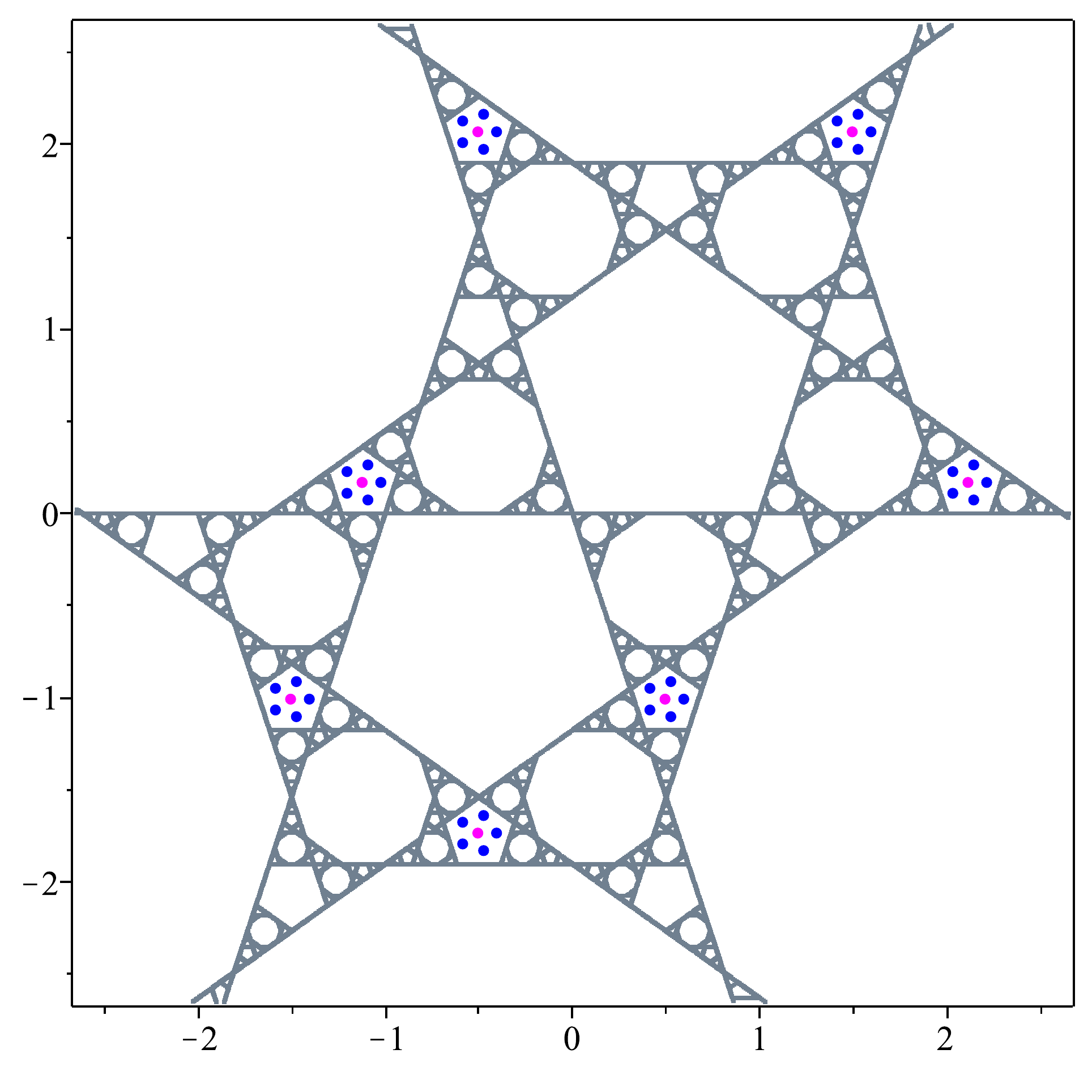}}
\caption{The $7$-periodic orbit associate with $P_1=r(P_0)$ in Magenta. All the points in the corresponding pentagons are $35$-periodic. In blue a $35$-periodic orbit.}\label{f:6}
\end{figure}
Also, it seems that ${\mathrm{period}(P_{n+1})}/{\mathrm{period}(P_{n})}\to 6$. The existence of this sequence of points implies the existence of nested connected components of $\mathcal{U}$ in bounded regions, filled by periodic points whose periods seem to increase indefinitely. This allows to find segments in $\mathcal{F}$ belonging to the boundaries of these connected components of  $\mathcal{U}$, filled by periodic points whose periods also seem to increase indefinitely.

\begin{figure}[H]
\centerline{\includegraphics[scale=0.35]{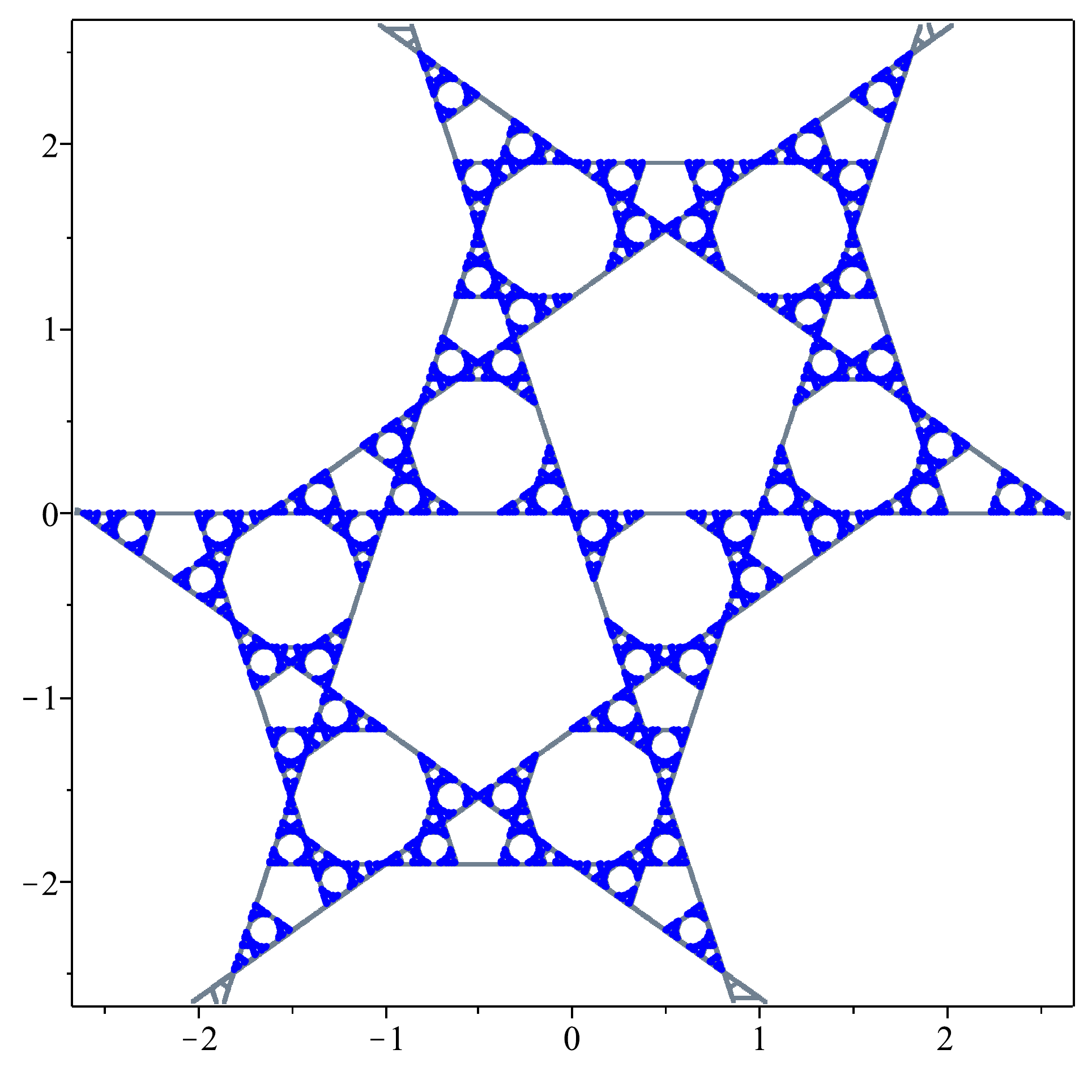}}
\caption{Some iterates of the orbit of the point  $Q=(-\varphi,0)\in\mathcal{F}$ in blue.}\label{f:7}
\end{figure}

It is also still open the possibility of the existence of aperiodic points in $\mathcal{F}$. In this sense, observe that the sequence of points $\{P_n\}_n$ converges to the point $Q=(-\varphi,0)\in\mathcal{F}$. This point has been identified as a possible aperiodic point in \cite{ChaCheYeh2019}. In Figure \ref{f:7}, the orbit of $Q$ is displayed.

We notice that we could compute, analytically, a very large number of iterates of $Q$ without finding a period. This is because, for this case $\lambda=\exp(i 8\pi/5)=\exp(-i 2\pi/5)=(\varphi-1)/2-i\,\sqrt{\varphi+2}/2$. Hence, 
by using the relation $\varphi^2=\varphi+1$, and setting $Q=-\varphi\in\mathbb{C}$, we obtain that:
$$
Q_n=F_\lambda^n(Q)=a_n+b_n\varphi+i\,\left(c_n+d_n\varphi\right)\sqrt{2+\varphi},$$ with $a_n,b_n,c_n,d_n\in\mathbb{Q}$, hence we could work with rational arithmetic. For instance, the first 10 iterates are: 
\begin{align*}
&Q_0=-\varphi,\,
Q_1= -\frac{\varphi}{2}+i\,\left(\frac{1}{2}+\frac{\varphi}{2}\right)\sqrt{2+\varphi} ,\,
Q_2= 1+\frac{\varphi}{2}+i\,\left(\frac{1}{2}+\frac{\varphi}{2}\right)\sqrt{2+\varphi},
Q_3= 1+\varphi,\\
&Q_4=\frac{1}{2}-i\,\frac{ \varphi  }{2}\sqrt{2+\varphi},\,
Q_5= -1-i\, \sqrt{2+\varphi},\,
Q_6= -1-\frac{\varphi}{2}+i\, \left(\frac{1}{2}-\frac{\varphi}{2}\right)\sqrt{2+\varphi},\\
&Q_7=-\frac{\varphi}{2}+i\, \left(\frac{ \varphi}{2}-\frac{1}{2}\right)\sqrt{2+\varphi},\,
Q_8= i\,  \sqrt{2+\varphi},\,
Q_9=\frac{3}{2}+i\,\frac{\varphi}{2}\,\sqrt{2+\varphi},\,
Q_{10}=\varphi.
\end{align*}

Interestingly, the points of the sequence that return to the critical line seem to have the form $a_n+b_n\varphi$ where $a_n, b_n\in\mathbb{Z}$. For instance the returns in the first $220$ iterates are: 
$Q_0=-\varphi,$ 
$Q_3= 1+\varphi,$ $Q_{10}=\varphi,$ 
$Q_{15}=-2 + \varphi,$ 
$Q_{38}=-3 + \varphi,$ 
$Q_{48}=-3 + 3 \varphi,$ 
$Q_{53}= -5 + 3 \varphi,$ 
$Q_{78}= -7 + 5 \varphi,$ 
$Q_{83}= -9 + 5 \varphi,$ 
$Q_{93}= -9 + 7\varphi,$ 
$Q_{220}=-10 + 7\varphi$.

\bigskip


\subsection*{Acknowledgements} The first, second and fourth authors are supported by
Ministry of Science and Innovation--State Research Agency of the
Spanish Government through grants PID2019-104658GB-I00  (first and
second authors) and MTM2017-86795-C3-1-P (fourth autor). They are also supported by the grant 2021-SGR-00113  from AGAUR. The second author is supported  by
grant Severo Ochoa and Mar\'{\i}a de Maeztu Program for Centers and Units of Excellence in R\&D (CEX2020-001084-M).
The third author acknowledges the group research recognition 2021-SGR-01039 from AGAUR. We thank our colleague Roser Guardia for the indications regarding the scale factor of the critical set that we mentioned in Section~\ref{s:evidences}.

\end{document}